\documentclass[11pt,a4paper]{article}
\usepackage{graphicx}
\usepackage{color}
\usepackage{amsmath,amsthm,amsfonts}

\usepackage[utf8]{inputenc}

\newcommand\R{\mathbb{R}}

\newcommand\calH{\mathcal{H}}

\newcommand\Curl{\mathop{\mathrm{curl}}}
\newcommand\curl{\mathop{\mathrm{curl}}}
\newcommand\dist{{\mathrm{dist}}}

\newcommand\SO{{\mathrm{SO}}}

\newcommand\dx{{\,dx}}

\newtheorem{theorem}{Theorem}[section]
\newtheorem{lemma}[theorem]{Lemma}

\numberwithin{equation}{section}
\usepackage{fancyhdr}
\newcommand{\M}{\mathcal{M}}

\newcommand{\dive}{\mathrm{div}\,}

\begin{document}
	\begin{center}
		{ \LARGE
		 {Sharp rigidity estimates for incompatible fields as consequence of the Bourgain Brezis div-curl result}
			\\[5mm]}
		{\today}\\[5mm]
		Sergio Conti,$^{1}$ and Adriana Garroni,$^{2}$
		\\[2mm]
		{\em $^1$ Institut f\"ur Angewandte Mathematik,
			Universit\"at Bonn\\ 53115 Bonn, Germany }\\
		{\em $^{2}$ Dipartimento di Matematica, Sapienza, Universit\`a di Roma\\
			00185 Roma, Italy}\\
		[4mm]

		\begin{minipage}[c]{0.8\textwidth}
In this note we show that a sharp rigidity estimate and a sharp Korn's inequality for matrix-valued fields whose incompatibility is a bounded measure
can be obtained as a consequence 
of a Hodge decomposition with critical integrability
due to Bourgain and Brezis.
		\end{minipage}
	\end{center}
	

	

	\section{Introduction}
	
The celebrated rigidity estimate of Friesecke James and M\"uller \cite{FrieseckeJamesMuellerCRAS2002,FrieseckeJamesMueller2002} 
is a fundamental tool for the analysis of variational, geometrically nonlinear models in elasticity. Indeed it allows to obtain compactness of the deformation gradient  (and therefore, for example,  linearization or dimension-reduction results) when the elastic energy has a degeneracy
due to frame indifference.
The linear counterpart of this estimate is the classical Korn inequality, which permits to deal with problems in linear elasticity where the energy, due to the approximation of the frame indifference, does not control the antisymmetric part of the displacement gradient, the so called infinitesimal rotations.

In the presence of plastic deformations, and then beyond the elastic regime mentioned above, the variational models should account for plastic slips that occur at microscopic scale and for the presence of topological defects which induce elastic distortion. Therefore, the relevant variable is not the gradient of a deformation but a more general matrix-valued field, the elastic strain, which may not be a gradient and may have a non trivial curl, possibly concentrated on low-dimensional sets.

%

This framework calls for a version of the rigidity result, and of its linear counterpart, for fields that may be incompatible, in the sense that are not gradients, but for which one knows that the curl is a bounded measure.
Such results do not seem to be available in the literature in dimension greater than two, and are stated in Theorems \ref{rigidityn} and \ref{kornn}.

In the sequel, given a matrix-valued function $\beta\in L^1(\Omega; \R^{n\times n})$, $\Curl\beta$ denotes the tensor-valued distribution whose rows are the curl, in the sense of distributions, of the rows of $\beta$, 
in the sense that $\Curl\beta_{ijk}:=\partial_k \beta_{ij}-\partial_j \beta_{ik}$,
moreover $\M(\Omega;\R^k)$ denotes the space of Radon measures in $\Omega$ with values in $\R^k$. In dimension 3, $\Curl\beta$ can as usual be identified with a matrix-valued distribution 
$\alpha_{il}:=\sum_{j,k}\epsilon_{ljk}\Curl\beta_{ijk}$.
The natural space is then the one obtained by the Sobolev-conjugate exponent $1^*:=n/(n-1)$.

\begin{theorem}\label{rigidityn}
	Given  an open, bounded, connected,
	and Lipschitz set $\Omega$ in $\R^n$, with $n\ge 2$, there exists a constant $C = C(n,\Omega) > 0$  such that for every
	$\beta\in L^1(\Omega;\R^{n\times n})$ with $\Curl \beta\in \M(\Omega;\R^{n\times n\times n})$, there exists a rotation $R\in \SO(n)$ such that
\begin{equation}\label{f-rigidityn}
	\|\beta-R\|_{L^{1^*}(\Omega)}\leq C\left(\|\dist(\beta,\SO(n))\|_{L^{1^*}(\Omega)}+|\Curl \beta|(\Omega)  \right).
\end{equation}
\end{theorem}

\begin{theorem}\label{kornn}
	Given an open, bounded, connected,
	and Lipschitz set $\Omega$ in $\R^n$, with $n\ge2$, there exists a constant $C = C(n,\Omega) > 0$ such that for every
	$\beta\in L^1(\Omega;\R^{n\times n})$ with $\Curl \beta\in \M(\Omega;\R^{n\times n\times n})$ there exists an antisymmetric matrix $A$ such that
	\begin{equation}\label{f-kornn}
	\|\beta-A\|_{L^{1^*}(\Omega)}\leq C\left(\|\beta+\beta^T\|_{L^{1^*}(\Omega)}+|\Curl \beta|(\Omega)  \right).
	\end{equation}	
\end{theorem}

In the above results $|\Curl \beta|(\Omega) $ denotes the total variation of the measure $\Curl\beta$. In the case when $\Curl\beta$ is absolutely continuous with respect to the Lebesgue measure it is simply given by the $L^1$ norm of $\Curl\beta$. We remark that the main point in this estimate is that it provides
an estimate with the total variation and therefore allows for concentrated incompatibilities, which correspond to concentration of crystal defects.
Korn's inequalities with incompatibility involving the $L^p$ norm of $\Curl\beta$, with $p>1$, instead of the total variation, can be found in \cite{LewintanNeff2019necas} and the references therein.

In 2 dimensions we have $1^*=2$ and Korn and rigidity estimates for incompatible field (Theorems \ref{kornn}
and \ref{rigidityn} with $n=2$) have been obtained in \cite[Theorem~11]{GarroniLeoniPonsiglione2010} and \cite[Theorem~3.3]{MuellerScardiaZeppieri2014}, respectively. A key ingredient in the proof is an interpolation inequality due to Bourgain and Brezis which provides an estimate of the $H^{-1}$ norm of  a field in terms of the $L^1$ norm and the $H^{-2}$ norm of its divergence, \cite{BourgainBrezis2007}. 

In higher dimension a recent result by Lauteri and Luckhaus, \cite{LauteriLuckhaus2017}, shows that the rigidity estimate of Theorem \ref{rigidityn} can be obtained in the Lorentz space $L^{1^*,\infty}$, but their results do not cover the case $p=1^*$ in the results above.

Here, we show that the result in dimension three and above, Theorems \ref{rigidityn} and \ref{kornn}, can be obtained very easily via Hodge decomposition using a sharp regularity result due to Bourgain and Brezis for the div-curl system \cite {BourgainBrezis2004,BourgainBrezis2007} (see also \cite{LanzaniStein2005} and \cite{VanSchaftingen2004}). 

	\newcommand{\dHnu}{\,\,\mathrm{d}\calH^{n-1}}
	\section{Proof of the results for $n\ge 3$}
	
\begin{lemma}\label{lemmabb} Let $n\ge 3$ and $Y\in L^1(Q;\R^n)$ be such that $\dive Y=0$ distributionally and $Yn=0$ on $\partial Q$. Then
 \begin{equation}
  \| Y\|_{L^{1^*}(Q)} \le c(n) \|\curl Y\|_{L^1(Q)}.
 \end{equation}
\end{lemma}
Here and below $Q:=(0,1)^n$; the assumption on $Y$ means that extending $Y$ by zero to $\R^n$ we obtain a divergence-free field.

We recall that this assertion does not hold for $n=2$, as one can see with $\curl Y=\delta_{x_0}$ for some $x_0\in Q$, see for example \cite[Remark~2]{BourgainBrezis2004}. Indeed, the proof of Theorem \ref{rigidityn} and Theorem \ref{kornn} for $n=2$ given 
in \cite[Theorem~11]{GarroniLeoniPonsiglione2010} and \cite[Theorem~3.3]{MuellerScardiaZeppieri2014}
is different, although it also uses in a crucial way the construction by Bourgan and Brezis.

We remark that for $n=3$ this implies that for every $f\in L^1(\Omega;\R^3)$, with $\dive\, f=0$, there exists a unique solution $Z\in L^{3/2}(Q;\R^3)$ of the system
\begin{equation}\label{f-div-curl}
	\begin{cases}
	\dive\, Z= 0 & \hbox{in } Q,\\
	\curl\, Z = f & \hbox{in } Q,\\
	Z\cdot n=0 &\hbox{on }\partial Q
	\end{cases}
\end{equation}
	and it satisfies
	\begin{equation}\label{f-stima}
	\|Z\|_{L^{3/2}(Q)}\leq C \|f\|_{L^1(Q)}
	\end{equation}
	for some constant independent on $f$. 	By an approximation argument it follows that the same result holds for measures, i.e., $f\in\M(Q;\R^3)$, with $\dive f=0$ in the sense of distributions. By  \cite[Theorem~4.1]{BrezisVanschaftingen2007}  this result can be extended to smooth bounded sets $\Omega\subset\R^n$. This argument was used for the study of geometrically linear dislocations in \cite{ContiGarroniOrtiz2015}.
\begin{proof}
By density, it suffices to prove the bound for regular fields $Y$ (one can reflect $Y$ across all boundaries and then mollify).
 Let $X\in L^n(Q;\R^n)$. By Corollary 18' in \cite{BourgainBrezis2007} there are $\varphi\in W^{1,n}(Q)$ and $\gamma\in W^{1,n}_0\cap L^\infty(Q;\R^{n\times n})$ such that
 $\gamma+\gamma^T=0$, $X=D\varphi+\dive \gamma$, and
 \begin{equation}
\|\varphi\|_{W^{1,n}(Q)}+\|\gamma\|_{L^\infty(Q)}+\|\gamma\|_{W^{1,n}(Q)}\le c(n)\|X\|_{L^n(Q)}.
\end{equation}
We estimate
\begin{equation}
\begin{split}
 \int_Q Y\cdot X \dx &= \int_Q Y\cdot D\varphi \dx + \int_Q Y\cdot \dive\gamma \dx \\&=  \int_{\partial Q} Y\cdot \gamma n \dHnu -\int_Q DY\cdot \gamma \dx\\
 &=-\frac12\int_Q (DY-DY^T)\cdot \gamma dx = -\frac12 \int_Q \curl Y \cdot \gamma \dx \\
 &\le \frac12 \|\curl Y\|_{L^1(Q)}\|\gamma\|_{L^\infty(Q)}
 \le c(n)\|\curl Y\|_{L^1(Q)}\|X\|_{L^n(Q)}.
\end{split}
 \end{equation}
\end{proof}

\begin{proof}[Proof of Theorem \ref{rigidityn}]
To simplify notation, we write $s:=1^*=n/(n-1)$
and denote by $c$ a constant that depends only on the spatial dimension $n$, and that may change from line to line.
It suffices to prove the assertion if the right-hand side is finite.

First we prove the result in the case that $\Omega=Q=(0,1)^n$. 

By density it suffices to prove the assertion for $\beta$ smooth. Indeed, it is sufficient to extend $\beta$ by reflection and then to mollify.
We define $u\in W^{2,1}(Q;\R^n)$ as the solution of $\Delta u=\dive \beta$ in $Q$, $\partial_n u=\beta n$ on $\partial Q$, and set $Y:=\beta-Du$. Since $\dive Y=0$, $Yn=0$ on $\partial Q$, and $\curl Y=\curl\beta$, application of Lemma \ref{lemmabb} to each row of $Y$ yields
\begin{equation}
 \|Y\|_{L^{s}(Q)} \le c \|\curl \beta\|_{L^{1}(Q)}.
\end{equation}
Therefore, since $Du=\beta-Y$,
\begin{equation}
\begin{split}
\| \dist(Du, \SO(n))  \|_{L^{s}(Q)} &\le 
\| \dist(\beta, \SO(n))  \|_{L^{s}(Q)} 
+\| Y  \|_{L^{s}(Q)} 
\\
&
\le \| \dist(\beta, \SO(n))  \|_{L^{s}(Q)} +c \|\curl \beta\|_{L^{1}(Q)},
\end{split}
\end{equation}
and by the geometric rigidity estimate by Friesecke, James and M\"uller \cite{FrieseckeJamesMuellerCRAS2002,FrieseckeJamesMueller2002} 
(see also  \cite{ContiDolzmannMueller2014} for the version in $L^p$)
we obtain the assertion.
By scaling invariance, we have shown that for any cube $Q_r:=x_*+(-r,r)^n$ and any $\beta:Q_r\to \R^{n\times n}$ there is $R\in \SO(n)$ such that
\begin{equation}\label{stimacubo}
 \|\beta-R\|_{L^{s}(Q_r)}\le c\left(
\| \dist(Du, \SO(n))  \|_{L^{s}(Q_r)} + \|\curl \beta\|_{L^{1}(Q_r)} \right).
 \end{equation}

We now deal with a generic bounded Lipschitz connected set $\Omega$. We consider a Whitney covering of $\Omega$ by countably many cubes $Q^j:=x_j+(-r_j,r_j)^n$ with finite overlap, so that
\begin{equation}\label{eqchiQomega}
 \chi_\Omega \le \sum_j \chi_{\hat Q^j}\le \sum_j \chi_{Q^j}\le c \chi_\Omega,
\end{equation}
where  $\hat Q^j:=x_j+(-\frac 12  r_j,\frac 12  r_j)^n$, 
and such that $r_j\le \dist(Q_j,\partial\Omega)\le c r_j$.
Let $\varphi_j\in C^\infty_c(Q^j)$ be a partition of unity subordinated to the cubes $Q^j$, and such that $|D\varphi_j|\le c/r_j$.

We apply the estimate \eqref{stimacubo} on each cube $Q^j$ and obtain rotations $R_j\in SO(n)$. 
If $Q_j\cap Q_k\ne\emptyset$, then by a triangular inequality we obtain
\begin{equation}\label{eqqjcapqk}
 \int_{Q_j\cap Q_k} |R_j-R_k|^s \dx \le c\int_{Q_j\cap Q_k} |\beta-R_j|^s \dx + c\int_{Q_j\cap Q_k} |\beta-R_k|^s \dx.
\end{equation}
We define $R:\Omega\to\R^{n\times n}$ by
$R:=\sum_j \varphi_j R_j$, and observe that
\begin{equation}\label{eqbetaR}
 \|\beta-R\|_{L^s(\Omega)}^s \le c\sum_j \|\beta-R_j\|_{L^s(Q_j)}^s.
\end{equation}
From $\sum_j\varphi_j=1$ we obtain $\sum_j D\varphi_j=0$ on $\Omega$, so that
\begin{equation}
 DR = \sum_j D\varphi_j R_j=\sum_j D\varphi_j (R_j-\beta) .
\end{equation}
Recalling that $\dist(Q_j,\partial\Omega)\le c r_j$ and \eqref{eqchiQomega},
\begin{equation}
\begin{split}
 \int_\Omega \dist^s(x,\partial\Omega) |DR|^s \dx &\le c \sum_j \int_{Q_j}  r_j^s |D\varphi_j|^s |\beta-R_j|^s \dx\\
 &
\le c \sum_j \int_{Q_j}  |\beta-R_j|^s \dx.
 \end{split}
\end{equation}
Since \eqref{stimacubo} holds in each cube $Q^j$,
\begin{equation}
\begin{split}
 \int_\Omega \dist^s(x,\partial\Omega) |DR|^s \dx &
 \le c\sum_j  \left[ \|\dist(\beta,\SO(n))\|^s_{L^s(Q^j)} +  (|\curl \beta|(Q_j))^s\right]\\
 &\le c \|\dist(\beta,\SO(n))\|^s_{L^s(\Omega)} \\
 &+ c (|\curl \beta|(\Omega))^{s-1} \sum_j |\curl \beta|(Q_j)\\
 &\le c \|\dist(\beta,\SO(n))\|^s_{L^s(\Omega)} + c (|\curl \beta|(\Omega))^s.
\end{split}
\end{equation}
By the weighted Poincar\'e inequality \cite[Theorem~8.8]{Kufner80} there is a matrix $R_*\in \R^{n\times n}$ such that
\begin{equation}
 \|R-R_*\|_{L^s(\Omega)}^s\le c
\int_\Omega \dist^s(x,\partial\Omega) |DR|^s \dx.
 \end{equation}
 Let now $\hat R\in \SO(n)$ be a matrix such that $|\hat R-R_*|=\dist(R_*,\SO(n))$. Then, recalling \eqref{eqbetaR},
 \begin{equation}
 \begin{split}
 |\Omega|^{1/s} \, |\hat R-R_*| &\le \|R_*-R\|_{L^s} + 
\|R-\beta\|_{L^s} + \|\dist(\beta ,\SO(n))\|_{L^s} \\
&\le
c \|\dist(\beta,\SO(n))\|_{L^s(\Omega)} + c |\curl \beta|(\Omega).
 \end{split}
 \end{equation}
\end{proof}

\begin{proof}[Proof of Theorem \ref{kornn}]
This follows by the same proof, replacing geometric rigidity in $L^{s}$ by Korn's inequality in $L^{s}$ (see, for example,
\cite{ContiDolzmannMueller2014} and references therein).

\end{proof}
	
	\section*{Acknowledgements}
	This work was partially funded by the Deutsche Forschungsgemeinschaft (DFG, German Research Foundation) through project 211504053 -- SFB 1060
	and project 390685813 -- GZ 2047/1, and by MIUR via PRIN 2017 Variational
	methods for stationary and evolution problems with singularities and interfaces, 2017BTM7SN.

%
%

	
	
	\bibliographystyle{alpha-noname}
	\bibliography{Co-Gar-20-preprint}

\end{document}